\documentclass[12pt, reqno]{amsart}
\usepackage{amssymb, amscd, hyperref}
\usepackage{amsthm}
\usepackage{graphicx}
\usepackage[all,2cell]{xy}
\usepackage{verbatim}
\usepackage{tikz}
\usepackage{xypic}
\usepackage{tikz-cd}
\usepackage{hyperref}
\usepackage{calc}
\usetikzlibrary{decorations.markings}
\tikzstyle{vertex}=[circle, draw, inner sep=0pt, minimum size=6pt]

\usetikzlibrary{matrix,arrows,decorations.pathmorphing}
\usepackage{mathrsfs}
\usetikzlibrary{calc, through}
\numberwithin{equation}{section}
\newtheorem*{theorem*}{Theorem}
\newtheorem*{corollary*}{\bf Corollary}
\newtheorem*{remark*}{\bf Remark}
\usetikzlibrary{matrix,arrows,decorations.pathmorphing}
\usepackage{lineno}
\newtheorem{theorem}{Theorem}[section]
\newtheorem{corollary}[theorem]{Corollary}
\newtheorem{definition}[theorem]{Definition}

\newtheorem{lemma}[theorem]{Lemma}

\newtheorem{proposition}[theorem]{Proposition}
\newtheorem{remark}[theorem]{Remark}
\usepackage{csquotes}

\title[Cohomology of line bundles on horospherical varieties]
{Cohomology of line bundles on horospherical varieties}

 \author{Narasimha Chary Bonala}
\address{%
Narasimha Chary Bonala\\
Max Planck Institute for Mathematics\\
Vivatsgasse 7, Bonn\\
Germany
}
\curraddr{
Fakult\"{a}t f\"{u}r Mathematik\\
Ruhr-Universit\"{a}t Bochum, Bochum\\
Germany\\
Email: Narasimha.bonala@rub.de
}
\author{Beno\^{i}t Dejoncheere}
\address{ 
Beno\^{i}t Dejoncheere,
    University of Alberta,
    Department of Mathematical
    and Statistical Sciences,
    Edmonton, AB, Canada T6G 2G1.
    Email: dejonche@ualberta.ca }

\begin{document}
\maketitle

\begin{abstract} 
A horospherical variety is a normal algebraic variety where a connected
reductive algebraic group acts with an open orbit isomorphic to a torus bundle over a flag variety.
In this article we study the cohomology of line bundles on complete horospherical varieties.

\end{abstract}
 
\tableofcontents
\section{Introduction}

Let $G$ be a connected reductive algebraic group over the field of complex numbers $\mathbb C$ and let $H$ be a closed subgroup of $G$.
A homogeneous space $G/H$ is said to be horospherical if $H$ contains the unipotent radical of a Borel subgroup of $G$, or equivalently, $G/H$ is isomorphic to a torus bundle over a flag variety $G/P$.
A normal $G$-variety is called horospherical if it contains an open dense $G$-orbit isomorphic to a horospherical homogeneous space $G/H$.
Toric varieties and flag varieties are horospherical varieties (see \cite{Pa} for more details).
Horospherical varieties are a special class of spherical varieties. A spherical variety is a normal $G$-variety $X$ such that there exists $x\in X$ and a Borel subgroup $B$ of $G$ satisfying that the
$B$-orbit of $x$ is open in $X$.

Let $\mathcal L$ be a line bundle on $X$. Up to replacing $G$ by a finite cover, we can assume that $G=C\times [G, G]$, where $C$ is a torus and $[G, G]$ is simply connected, hence one can assume that $\mathcal{L}$ is $G$-linearized. If $X$ is complete, 
then the cohomology
groups $H^i(X, \mathcal L)$ are finite dimensional representations of $G$.
If $X$ is a flag variety, then Borel-Weil-Bott theorem describe these cohomology groups. 
If $X$ is a toric variety, then the cohomology groups can be described using combinatorics of its associated fan (see \cite{Cox}).
For spherical varieties, in \cite{brion94}, Brion has given a bound on the mulplicities of these modules. In \cite{Brion1990}, if $X$ is projective spherical and $\mathcal L$ is generated by global sections,
Brion proved that the cohomology groups $H^i(X, \mathcal L)$ for $i>0$ vanish. In \cite{Tchoudjem1}, \cite{Tchoudjem2}, \cite{Tchoudjem3} and \cite{Tchoudjem4},
Tchoudjem studied these cohomology groups in the case of X is a compactification of adjoint semisimple
group, a wonderful compactification of a reductive group, a wonderful variety of minimal rank and
a complete symmetric variety respectively. In this article we consider the cohomology of line bundles on complete horospherical varieties.
We prove that the cohomology groups splits into cohomology of line bundles on flag varieties and toric varieties.
Our main tool in this article is the machinery of Grothendieck-Cousin complexes (see \cite{Ke1} for more details), and we also prove a K\"{u}nneth-like formula for local cohomology. 

To describe our results we need some notation: 
let $X$ and $Y$ be affine schemes. Let $Z_1\subset X$ and $Z_2\subset Y$ be two locally complete intersections subschemes, of respective codimensions $l_1$ and $l_2$. 
Let $L_1$ and $L_2$ be two invertible sheaves respectively on $X$ and $Y$. Let $p_1:X\times Y \to X$ and $p_2:X\times Y\to Y$ be the projections, and 
let ${\mathcal L}_i:=p_i^*L_i$. Assume $X$ and $Y$ are irreducible and Cohen-Macaulay.
Then we have the following isomorphism (See Theorem \ref{kunneth-local}):
\begin{theorem*}
 $
H^{l_1+l_2}_{Z_1\times Z_2}(X\times Y, {\mathcal L}_1\otimes{\mathcal L}_2)\simeq H^{l_1}_{Z_1}(X,L_1)\otimes H^{l_2}_{Z_2}(Y,L_2).
$
\end{theorem*}

Let $X$ be a complete horospherical variety. We have that any Weil divisor $D$ of $X$ is linearly equivalent to a $B$-stable divisor, that is of the form (need not be unique)  
$$D\approx \sum_{i=1}^kd_iX_i+\sum_{\alpha\in I}d_{\alpha}D_{\alpha},$$
where $X_i'$s are $G$-stable prime divisors, $D_{\alpha}'$s are $B-$stable prime divisors and $I$ is a subset of the set of simple roots for $G$ (for more details see Section 2).
If we further assume that $X$ is smooth and toroidal, then $D\approx p^*E_1 + E_2$, where $E_1\approx \sum\limits_{\alpha\in I}d_{\alpha}D_{\alpha}$ is a divisor on the base space $G/P$, $p:X\to G/P$ is the projection map, and
$E_2\approx \sum\limits_{i=1}^kd_iX_i$ can be seen as a divisor coming from  the toric fibre $Y$, that is $X_i=G\times^PY_i$ where $Y_i$ is a toric prime divisor in $Y$ (see Section 2 for precise details). 
Let $T$ be a maximal torus of a Borel subgroup $B$ and $T'$ be 
the torus correspond to the toric fibre $Y$. Let $E_2'=\sum_{i=1}^kd_iY_i$.
Let $\mathfrak{g}$ be the Lie algebra of $G$. 
Then we prove the following (See Corollary \ref{horospherical}):
\begin{corollary*}
Let $X$ be a smooth toroidal complete variety. There is an  isomorphism (non canonical) of $T\times T'$-modules and of $\mathfrak{g}$-modules
\[
H^n(X, \mathcal O _X(D))\simeq H^n(X,\mathcal{O}_X(p^*E_1 + E_2)) = \bigoplus\limits_{p+q=n}H^p(G/P,{\mathcal O}_{G/P}(E_1))\otimes H^q(Y,{\mathcal O}_Y(E'_2))
\]
\end{corollary*}
\begin{remark*}
The $T\times T'$-module and $\mathfrak{g}$-module structures on these cohomology groups are compatible. By Borel-Weil-Bott theorem, at most one of the factors in the previous direct sum is nonzero.
\end{remark*}
The case of other complete horospherical varieties is straightforward: they all admit a $G$-equivariant resolution by a smooth complete toroidal horospherical variety (see Proposition \ref{nontoroidalcase}).
We now give a more concrete description of the cohomology groups $H^n(X,\mathcal{O}_X(p^*E_1 + E_2))$.

Let $\varpi_{\alpha}$ be the fundamental weight associated to a simple root $\alpha$.
Let $W=W(G,T)$ be the Weyl group, let $W_P$ be the Weyl group of $P$, and $W^P=W/W_P$ the 
set of minimal length representatives of the quotient. 
Let $\rho$ be the half-sum of all positive roots, and define $w\star \lambda:=w(\lambda + \rho) -\rho$ for any $w\in W$ and $\lambda$ in the weight lattice $\Lambda$ of $T$. 
Note that the line bundle $\mathcal O_{G/P}(E_1)$ is the homogeneous line bundle 
on $G/P$ correspond to the weight $\sum\limits_{\alpha\in I}d_{\alpha}\varpi_{\alpha}$. 
Then the Borel-Weil-Bott theorem states that $H^i(G/P,\mathcal O_{G/P}(E_1))=0$ 
unless there exists a  $w\in W^P$ of length $i$ such that 
$w\star (\sum\limits_{\alpha\in I}d_{\alpha}\varpi_{\alpha})$ lies in the 
interior of the dominant chamber, and in that case $H^{i}(G/P, \mathcal O_{G/P}(E_1))$ is the dual 
of the irreducible 
heighest weight module $V_{w\star (\sum\limits_{\alpha\in I}d_{\alpha}\varpi_{\alpha})}$.

On the other hand, we can describe the cohomology of line bundles on the smooth projective 
toric fibre by a result of Demazure (we refer to  \cite[Chapter 9]{Cox} or \cite[Section 3.4]{Ful} 
for more details). 
Let $N$ be the lattice of one-parameter subgroups of the torus $T'=P/H$ and let $M$ be 
 the lattice of characters of $T'$. Let $M_{\mathbb R}:=M\otimes \mathbb R$ and 
 $N_{\mathbb R}:=N\otimes \mathbb R$. Then we have a natural bilinear pairing 
  $\langle-,-\rangle:M_{\mathbb R}\times N_{\mathbb R}\to \mathbb R.$
Let $\Delta$ be the fan of the toric variety $Y$. 
By completeness of $X$, $\Delta$ is a complete fan in $N_{\mathbb R}$, that is $\cup_{\sigma\in \Delta}\sigma =N_{\mathbb R}$. 
We denote by $\Delta(l)$ the set of $l$-dimensional cones in $\Delta$. 
Then we can parametrize $T'$-invariant prime divisors of $Y$ by the set $\Delta(1)$.
Since the cohomology spaces $H^i(Y,{\mathcal O}_Y(E'_2))$ are $T'$-modules, 
these spaces have a weight decomposition: 
$$H^i(Y,{\mathcal O}_Y(E'_2))=\bigoplus_{m\in M} H^i(Y,{\mathcal O}_Y(E'_2))_m$$
and by a theorem of Demazure its degree $m$ part is described as follows: Let 
$Z_{E'_2}(m):=\{v\in N_{\mathbb{R}}, \langle m,v\rangle \geq \psi_{E'_2}(v)\}$, 
where $\psi_{E_2'}$ is the support function corresponding to the divisor $E_2'$ 
(for definition of the support function of a divisor we refer to \cite[Theorem 4.2.12]{Cox}).
Then
\[
H^i(Y,{\mathcal O}_Y(E'_2))_m=H^i_{Z_{E'_2}(m)}(N_{\mathbb{R}},\mathbb{C}).
\] 
Now we can reformulate our result as follows:
\begin{corollary*}
Let $X$ be a smooth complete horospherical variety. Then
\item[1)] If there is no $w\in W^P$ such that $w \star (\sum\limits_{\alpha\in I}d_{\alpha}\varpi_{\alpha})$ lies in the interior of the dominant chamber, then $H^n(X,\mathcal{O}_X(p^*E_1 + E_2))=0$ for all $n$.
\item[2)] If there exists $w\in W^P$ such that $w \star (\sum\limits_{\alpha\in I}d_{\alpha}\varpi_{\alpha})$ lies in the interior of the dominant chamber, such a $w$ is unique, and as a $\mathfrak{g}\times T'$-module we have 
\[
H^{l(w)+q}(X,\mathcal{O}_X(p^*E_1 + E_2)) = V_{w\star (\sum\limits_{\alpha\in I}d_{\alpha}\varpi_{\alpha})}\otimes \bigoplus\limits_{m\in M}H^q_{Z_{E_2}(m)}(N_{\mathbb{R}},\mathbb{C})\]
and $H^i(X,\mathcal{O}_X(p^*E_1 + E_2))=0$ for $i<l(w)$.
\end{corollary*}

This paper is organized as followed. Section 2 will be a remainder on horospherical varieties. Section 3 is devoted to proving Theorem \ref{kunneth-local}. In Section 4, we shall recall the 
Cousin complexes and show how to compute, under nice assumptions, the cohomology of line bundles on some locally trivial fibrations. The behaviour of the obtained isomorphisms with respect to actions of 
algebraic groups will be investigated in Section 5, and we shall prove Corollary \ref{horospherical} in Section 6. It should be noted that results of Sections 3 to 5 are characteristic free.

\section{Preliminaries on horospherical varieties}

In this section we recall some basic definition and properties of complete horospherical varieties. 
For more details we refer to \cite{Pa} and \cite{perrin2014geometry}.

Let $G$ be a connected reductive algebraic group over the field of complex numbers. A $G$-variety is a reduced scheme of finite type over the field of complex numbers, with an algebraic action of $G$.
Let $B$ be a Borel subgroup of $G$.

\begin{definition}
 A spherical variety is a normal $G$-variety $X$ such that it has a dense open $B$-orbit.
\end{definition}

These varieties have many interesting properties, for example, there are only finitely many $B$-orbits (see \cite{brion1986quelques}).
Here, we consider a special class of spherical varieties, called horospherical varieties. 

\begin{definition}
 Let $H$ be a closed subgroup of $G$. Then $H$ is said to be horospherical if it contains the unipotent radical of a Borel subgroup of $G$.
 \end{definition}
We also say that the homogeneous space $G/H$ is horospherical if $H$ is horospherical. Up to conjugation, we can assume that $H$ contains $U$, the maximal unipotent subgroup of $G$ contained in 
the Borel subgroup $B$. 
Let $P$ be the normalizer $N_{G}(H)$ of $H$ in $G$. It is a parabolic subgroup of $G$, and $T':=P/H$ is a torus.
Then it is clear that $G/H$ is a torus bundle over the flag variety $G/P$ with fibers $P/H$.
\begin{definition}
 Let $X$ be a $G$-spherical variety and let $H$ be the stabilizer of a point in the dense $G$-orbit in $X$.
Then we say $X$ is horospherical if $H$ is horospherical subgroup of $G$.
\end{definition}

A geometric characterization can be given as follows (see for example, \cite{Pa}):
Let $X$ be a normal $G$-variety. Then $X$ is horospherical if there exists a parabolic subgroup $P$ of $G$, and a smooth toric $T'$-variety $Y$ such that there is a diagram of $G$-equivariant morphisms 
\begin{center}
\begin{tikzcd}
G\times^P Y \arrow{d}{p}\arrow{r}{\pi}	& X\\
G/P		&
\end{tikzcd}
\end{center}
\noindent where $G\times^P Y:=(G\times Y)/P$, the action of $P$ is given by $p\cdot (g, y)=(gp, p^{-1}\cdot y)$, $\pi$ is birational and proper (ie. a resolution of singularities), and $p$ is a locally trivial fibration with fibers isomorphic to $Y$.
The classification of horospherical varieties can be given by using colored fans, see \cite{Pa} for precise details.
\subsection{Divisors of horospherical varieties}
In this subsection we recall some properties of divisors of a spherical variety $X$. We start by the following result from \cite{brion1989groupe}:
\begin{proposition}
 Any divisor of $X$ is linearly equivalent to a $B$-stable divisor. 
\end{proposition}
 Let $X$ be a  horospherical variety. Now we consider the $B$-stable prime divisors of $X$. We denote by $X_1,X_2,\ldots, X_r$ the $G$-stable divisors of $X$. The other $B$-stable divisors 
 (i.e. those are not $G$-stable) are the closures of $B$-stable divisors 
 of $G/H$, which are called the colors of $G/H$. These are the inverse images of the torus fibration $G/H\longrightarrow G/P$ of the Schubert divisors of the flag variety $G/P$.
 Now we recall the description of the Schubert divisors of $G/P$. Fix a maximal torus $T$ of $B$. Then we denote by $S$ the set of simple roots of $G$ with respect to $B$ and $T$.
Also denote by $S_P$ the subset of $S$ of simple roots of $P$, that is simple roots of the Levi factor $L_P$ of $P$.   
Then the Schubert divisors of $G/P$ are indexed by the subset of simple roots $S\setminus S_P$ and are of the form 
$\overline{Bw_0s_{\alpha}P/P}$, where $w_0$ denotes the longest element of the Weyl group of $G$ and $s_{\alpha}$ denotes the simple reflection associated to the simple root $\alpha$.
Hence the $B$-stable irreducible divisors of $G/H$ are of the form $\overline{Bw_0s_{\alpha}P/H}$, which we denote by $D_{\alpha}$ for $\alpha\in S\setminus S_P$. 
Then the following holds:
\begin{lemma}
 Any divisor $D$ of $X$ equivalent to a linear combination of $X_i$'s and $D_{\alpha}$ for $\alpha\in S\setminus R$. That is,
 $$D\approx \sum_{i=1}^rd_iX_i+\sum_{\alpha\in S\setminus S_P}d_{\alpha}D_{\alpha}.$$
\end{lemma} 

A horospherical variety is called toroidal if it has no color containing a $G$-orbit. It is of the form $G\times^P Y$, where $Y$ is as above. We summarize the desciption of smooth toroidal varieties that will be 
used later with the following proposition (and we refer to \cite[prop. 2.2 and ex. 2.3]{Pa} for its proof):
\begin{proposition}
Let $X$ be a smooth toroidal horospherical variety, and let $H$ be the stabilizer of a point in the open $B$-orbit of $X$.
\item (i) There is a parabolic subgroup $P$ containing $B$ such that 
\[
H=\bigcap\limits_{\lambda\in {\mathcal X}(P)}\ker(\lambda)
\]
where ${\mathcal X}(P)$ denotes the characters of $P$. Moreover, $P=N_G(H)$, and $T':= P/H$ is a torus.
\item (ii) There is a smooth $T'$-toric variety $Y$ such that $X=G\times^P Y$, and the natural map $X=G\times^P Y\to G/P$ is a locally trivial fibration, where the action of $P$ on $Y$ via $P\to T'$.
\item (iii)  $G\times T'$ acts on $X=G\times^P Y$ via 
\begin{equation}\label{action}
 (g,pH).[h, y]= [gh, p\cdot y]
\end{equation}
  for all $g, h \in G$ and $y\in Y$.

\end{proposition}

\subsection{Local structure of toroidal varieties}
In this subsection, we recall the local structure of a toroidal horospherical
variety $X$. We keep the notations of the previous subsection. 
Let $P^-$ be the parabolic subgroup opposite to $P$, and let $L_P=P\cap P^-$ be the Levi factor of $P$ containing $T$. 
Notice that there is a quotient map $L_P\to T'$, whose kernel is denoted by $L_0$. 
We denote by ${\mathcal D}(X)$ the set of colors of $X$, and define
\[
X_0:=X\setminus\bigcup\limits_{D\in{\mathcal D}(X)}D
\]
Then $X_0$ is a $P$-stable open subset such that $X$ is covered by the 
$G$-translates of $X_0$, and the local structure theorem describes precisely $X_0$:
\begin{proposition}\label{local_structure}\

\begin{enumerate}
 \item 
There exists a closed $L_P$-stable subvariety $Z$ of $X_0$, fixed pointwise by $L_0$, such that
\[
X_0\simeq P^-\times^{L_P} Z\simeq R^u(P^-)\times Z
\]
where $R^u(P^-)$ is the unipotent radical of $P^-$. Moreover, $Z$ is isomorphic to the $T'$-toric variety $Y$, 
it is defined by the same fan as $X$, and any $G$-orbit intersect $Z$ along a unique $T'$-orbit.
\item   The action of $B\times T'$ on $X=G\times^P Y$ (see \ref{action}) stabilizes $X_0$ and, the above isomorphism is $B\times T'$-equivariant (where the action on $R^u(P^-)\times Z$ is the product action).  
\end{enumerate}
\end{proposition}
 For a proof we refer to \cite[Proposition 3.4]{brion1987valuations} and also see \cite[Theorem 29.1]{Ti}. 
 Note that in particular, the fibration $P^-\times^{L_P} Z$ is trivial, 
 and that any $G$-stable divisor of $G\times^P Y$ is of the form $G\times^P D'$, 
 where $D'$ is a $T'$-stable divisor of $Y$. From now on, assume that $G\simeq C\times [G,G]$, where $C$ is a torus.
 
 We also have an exact sequence $0\to \textnormal{Pic}(G/P)\to \textnormal{Pic}(X)\to\textnormal{Pic}(Y)\to 0$ (see Prop \ref{Pic}).

Since both $G$ and $P$ are factorial, and since  both $G/P$ and $Y$ are normal, we have the following exact sequences (see \cite[Remark 2.4]{KKLV} and \cite[Proposition 2.10]{brion-lin})
$$0\to {\mathcal X}(G)\to \textrm{Pic}^G(G/P)\to \textrm{Pic}(G/P)\to 0$$ and 
\begin{equation}\label{1}
  0\to {\mathcal X}(P) \to \textrm{Pic}^P(Y)\to \textrm{Pic}(Y)\to 0
  \end{equation}
First observe that $\textrm{Pic}^G(G/P)=\mathcal X(P)$ and by above discussion we can see that $$\textrm{Pic}^G(X)=\textrm{Pic}^G(G\times^P Y)=\textrm{Pic}^P(Y).$$
The natural morphism $X\to G/P$  induces a map $\textrm{Pic}^G(G/P)\to \textrm{Pic}^G(X)$ via pullback.
Then by  exact sequence (\ref{1}) we have the following short exact sequence: 
$$0\to \textrm{Pic}^G(G/P) \to \textrm{Pic}^G(X)\to \textrm{Pic}(Y) \to 0.$$
Since $Y$ is a toric variety, the $\mathbb{Z}$-module $\textrm{Pic}(Y)$ is projective. Hence the above short exact sequence splits and we have 
 $$\textrm{Pic}^{G}(X)\simeq \textrm{Pic}^G(G/P)\oplus\textrm{Pic}(Y)$$

The same section provides a canonical splitting
\begin{equation}\label{equi}
\textrm{Pic}^{G\times T'}(X)=\textrm{Pic}^{G}(G/P)\oplus\textrm{Pic}^{T'}(Y)
\end{equation}

\section{A K\"{u}nneth-like formula for local cohomology}
We start by recalling a few results on local cohomology functors. We refer the reader to \cite{SGA2} and \cite{Ke1} for more details. 
In Sections 3 to 5, for simplicity and unless noted otherwise, all schemes will be supposed to be Noetherian over an algebraically closed field $k$.\\
\indent Let $X=\textrm{Spec}(A)$ be an affine scheme, let $QCoh(X)$ be the category of quasi-coherent ${\mathcal O}_X$-modules, and let $Z\subset X$ be a closed subscheme.
\begin{proposition}
Let ${\mathcal F}\in QCoh(X)$, and $p\geq 0$. The functor ${\mathcal H}^p_Z({\mathcal F}\otimes \bullet) : QCoh(X) \to QCoh(X)$ is additive and commutes with direct limits.
\end{proposition}
\begin{proof}
Since we assume $X$ to be Noetherian, all open $U\subset X$ are quasi-compact. 
Then to prove the proposition, one can use the same arguments as in \cite[Section 2]{Ke2}, 
where this is done for $H^p(X,\bullet)$. 
One should recall that, as for ordinary sheaf cohomology, one can compute $H^p_Z(X,\bullet)$ using a flabby resolution.
\end{proof}
We can now prove the following useful lemma.
\begin{lemma}{\label{isomorphism-flat-modules}}
Let ${\mathcal F},{\mathcal G}\in QCoh(X)$. Let $t^p_{{\mathcal F},{\mathcal G}} : {\mathcal H}^p_Z({\mathcal G})\otimes_{{\mathcal O}_X}{\mathcal F}\to 
{\mathcal H}^p_Z({\mathcal G}\otimes_{{\mathcal O}_X}{\mathcal F})$ be the natural map, and let $l$ be the codimension of $Z$ in $X$.
\item (i) If ${\mathcal F}$ is a flat ${\mathcal O}_X$-module, $t^p_{{\mathcal F},{\mathcal G}}$ is an isomorphism for all $p\geq 0$.
\item (ii) Assume that $X$ is Cohen-Macaulay. If ${\mathcal G}$ is a flat ${\mathcal O}_X$-module, and if $Z$ is a locally complete intersection, $t^l_{{\mathcal F},{\mathcal G}}$ is an isomorphism.
\end{lemma}
\begin{proof}
Proof of (i): We first recall the following facts:
\item (1) If $T: \{\textrm{left} A-\textrm{Mod}\}\to {\mathcal A}b$ is an additive functor, and if $L$ is a free $A$-module, then
\[
T(A)\otimes_A L \to T(L)
\]
is an isomorphism (see \cite[Lemma 7.2.4]{EGA3})
\item (2) (Lazard's theorem) Let $M$ be a $A$-module. Then $M$ is flat if and only if  it is the colimit of a directed system of free finite $A$-modules.\\
Then by flatness of ${\mathcal F}$, ${\mathcal F}(X)$ is flat over $A$, hence ${\mathcal F}(X)$ is a directed colimit of free $A$-modules of finite rank, and ${\mathcal F}$ is a directed 
colimit of free ${\mathcal O}_X$-modules of finite rank. Now since both ${\mathcal H}^p_Z({\mathcal G}\otimes \bullet)$ and global sections commute with direct limits, it is enough to show 
the isomorphism for global sections of free ${\mathcal O}_X$-modules, which is exactly (1).\\
\indent
Proof of (ii):We start by proving that the functor ${\mathcal H}^l_Z({\mathcal G}\otimes \bullet)$ is right exact. Let 
\[
0\to{\mathcal F}_1\to{\mathcal F}_2\to{\mathcal F}_3\to 0
\]
be an exact sequence of quasi-coherent ${\mathcal O}_X$-modules. Since ${\mathcal G}$ is flat, this gives a long exact sequence of local cohomology sheaves
\[
\ldots \to {\mathcal H}^l_Z({\mathcal G}\otimes{\mathcal F}_1)\to {\mathcal H}^l_Z({\mathcal G}\otimes{\mathcal F}_2)\to {\mathcal H}^l_Z({\mathcal G}\otimes{\mathcal F}_3)\to 
{\mathcal H}^{l+1}_Z({\mathcal G}\otimes{\mathcal F}_1)
\]
But the conditions on $Z$ ensure the vanishing of ${\mathcal H}^{l+1}_Z({\mathcal G}\otimes{\mathcal F}_1)$ (see \cite[Lemma 3.12, III]{SGA2}), hence the right-exactness.\\
Now pick an exact sequence ${\mathcal L}'\to {\mathcal L}\to {\mathcal F}\to 0$ with ${\mathcal L}$ and ${\mathcal L}'$ being free ${\mathcal O}_X$-modules. We have a commutative diagram
\begin{center}
\begin{tikzcd}
{\mathcal H}^l_Z({\mathcal G})\otimes{\mathcal L}' \arrow{d}\arrow{r}	& {\mathcal H}^l_Z({\mathcal G})\otimes{\mathcal L} \arrow{d}\arrow{r}	& {\mathcal H}^l_Z({\mathcal G})\otimes{\mathcal F} \arrow{d}\arrow{r}	& 0 \\
{\mathcal H}^l_Z({\mathcal G}\otimes{\mathcal L}') \arrow{r}			& {\mathcal H}^l_Z({\mathcal G}\otimes{\mathcal L}) \arrow{r}	& {\mathcal H}^l_Z({\mathcal G}\otimes{\mathcal F}) \arrow{r}	& 0 \\
\end{tikzcd}
\end{center}
Since ${\mathcal L}$ and ${\mathcal L}'$ are free, the two first vertical maps are isomorphisms, so is the third one.
\end{proof}
\begin{remark}
Any invertible sheaf ${\mathcal L}$ on $X= \textrm{Spec}(A)$ is flat.
\end{remark}
Before stating our K\"{u}nneth-like formula for local cohomology, let us recall a few known results on local cohomology (which hold in much more generality).
\begin{lemma}
Let ${\mathcal F}\in QCoh(X)$.
\item (i) If $Z_2\subset Z_1$ are two closed subschemes of $X$ such that $Z_1\setminus Z_2$ is affine, then
\[
H^i_{Z_1/Z_2}(X,{\mathcal F}) \simeq H^0(X,{\mathcal H}^i_{Z_1/Z_2}({\mathcal F}))
\]
\item (ii) Let $Z_2\subset Z_1$, $W_2\subset W_1$ be closed subschemes, and let $S_1=W_1\cap Z_1$, $S_2=(W_1\cap Z_2)\cup (W_2\cap Z_1)$. Then we have a spectral sequence
\[
E_2^{p,q}={\mathcal H}^p_{W_1/W_2}({\mathcal H}^q_{Z_1/Z_2}({\mathcal F}))\Rightarrow {\mathcal H}^*_{S_1/S_2}({\mathcal F})
\]
In particular, when $Z_2=W_2=\emptyset$, we have
\[
E_2^{p,q}={\mathcal H}^p_{W_1}({\mathcal H}^q_{Z_1}({\mathcal F}))\Rightarrow {\mathcal H}^*_{W_1\cap Z_1}({\mathcal F})
\]
\item (iii)  Let $Z_2\subset Z_1$ be two close subschemes of $X$. Let $Y$ be scheme, and $p:X\times Y\to X$ be the first projection. There is an isomorphism
\[
p^*{\mathcal H}^i_{Z_1/Z_2}({\mathcal F})\simeq {\mathcal H}^i_{Z_1\times Y/Z_2\times Y}(p^*{\mathcal F})
\]
\end{lemma}
\begin{proof}
 For proof of (i) see \cite[Theorem 9.5 (d)]{Ke1}. For proof of (ii) see \cite[Lemma 8.5 (d)]{Ke1}. For proof of (iii) see \cite[Proposition 11.5]{Ke1}.
\end{proof}

Let $Y$ be an affine scheme. Let $Z_1\subset X$ and $Z_2\subset Y$ be two locally complete intersections, of respective codimensions $l_1$ and $l_2$. 
Let $L_1$ and $L_2$ be two invertible sheaves respectively on $X$ and $Y$. Let $p_1:X\times Y \to X$ and $p_2:X\times Y\to Y$ be the projections, and 
let ${\mathcal L}_i:=p_i^*L_i$. We are now ready to state our result.
\begin{theorem}\label{kunneth-local}
Assume $X$ and $Y$ are irreducible and Cohen-Macaulay. Then
\[
H^{l_1+l_2}_{Z_1\times Z_2}(X\times Y, {\mathcal L}_1\otimes{\mathcal L}_2)\simeq H^{l_1}_{Z_1}(X,L_1)\otimes_k H^{l_2}_{Z_2}(Y,L_2)
\]
\end{theorem}
\begin{proof}
First let us notice that
\[
H^{n}_{Z_1\times Z_2}(X\times Y,{\mathcal L}_1\otimes {\mathcal L}_2)=H^0(X\times Y,{\mathcal H}^n_{Z_1\times Z_2}({\mathcal L}_1\otimes{\mathcal L}_2))
\]
Since ${\mathcal L}_1$ and ${\mathcal L}_2$ are locally free, and by the conditions on $Z_i$'s, the spectral sequence
\[
E^{p,q}_2 = {\mathcal H}^p_{Z_1\times Y}({\mathcal H}^q_{X\times Z_2}({\mathcal L}_1\otimes {\mathcal L}_2))\Rightarrow {\mathcal H}^*_{Z_1\times Z_2}({\mathcal L}_1\otimes {\mathcal L}_2)
\]
collapses, and it may only be nonzero for $p=l_1$ and $q=l_2$. Here we are using that ${\mathcal L}_1\otimes {\mathcal L}_2$ is also locally free, and that 
${\mathcal I}_{Z_1\times Z_2}$ is locally generated by $l_1+l_2$ elements.\\
\indent
Using Lemma \ref{isomorphism-flat-modules}, we compute
\[
\begin{array}{rclr}
{\mathcal H}^p_{Z_1\times Y}({\mathcal H}^q_{X\times Z_2}({\mathcal L}_1\otimes {\mathcal L}_2)) & \simeq & {\mathcal H}^p_{Z_1\times Y}({\mathcal L}_1\otimes {\mathcal H}^q_{X\times Z_2}({\mathcal L}_2)) & (\textrm{by Lemma \ref{isomorphism-flat-modules} (i))}\\
& \simeq & {\mathcal H}^p_{Z_1\times Y}({\mathcal L}_1)\otimes {\mathcal H}^q_{X\times Z_2}({\mathcal L}_2) & (\textrm{by Lemma \ref{isomorphism-flat-modules}  (ii))}
\end{array}
\]
Taking global sections yield the wanted isomorphism.
\end{proof}

\section{Locally trivial fibrations and Cousin complexes}
Let $S$ and $F$ be smooth complete connected schemes, and let $E\to S$ be a (Zariski) locally trivial
fibration, with fiber $F$. Moreover, assume that the Picard group Pic$(F)$ of $F$ is a projective $\mathbb{Z}$-module, and that $F$ is rational. We start with the following useful 
description of Picard group Pic$(E)$ of $E$.
\begin{proposition}\label{Pic}
We have an (noncanonical) isomorphism $\textnormal{Pic}(E)\simeq \textnormal{Pic}(S)\oplus \textnormal{Pic}(F)$
\end{proposition}
\begin{proof}
Using \cite[Proposition 2.3]{FI}, we have an exact sequence
\[
H^0(F,{\mathcal O}_F^*)/k^* \to \textrm{Pic}(S) \to \textrm{Pic}(E) \to \textrm{Pic}(F) \to 0
\]
By the hypotheses on $F$, $H^0(F,{\mathcal O}_F^*)=k^*$. Since $\textrm{Pic}(F)$ is projective, 
we get $\textnormal{Pic}(E)\simeq \textnormal{Pic}(S)\oplus \textnormal{Pic}(F)$. This completes the proof of the proposition. 
\end{proof}

Let $\mathcal L$ be a line bundle on $E$, then by Lemma \ref{Pic}, we can identify $\mathcal L$ with $\mathcal L_1\otimes \mathcal L_2$ for some line bundles ${\mathcal L}_1$ and ${\mathcal L}_2$ on $S$ and $F$ respectively. 
The aim of this section is to compute the cohomology groups $H^n(E,\mathcal L)$. 
To do that, we shall need the formalism of Cousin complexes.\\
\indent
Let $X$ be a Noetherian scheme, let ${\mathcal F}$ be a ${\mathcal O}_X$-module, and let $\{Z\}:=\emptyset=Z_{n+1}\subset Z_n\subset \ldots \subset Z_1\subset Z_0=X$ be a filtration by 
closed subschemes. For simplicity, let us assume that $X$ is irreducible. Then one can construct (see \cite[Lemma 7.8]{Ke1}) the Cousin complex of ${\mathcal F}$ relatively to the filtration $\{Z\}$
\[
\textrm{Cousin}_{\{Z\}}({\mathcal F}) := 0\to H^0(X,{\mathcal F})\to H^0_{Z_0/Z_1}(X,{\mathcal F})\to H^1_{Z_1/Z_2}(X,{\mathcal F}) \to H^2_{Z_2/Z_3}(X,{\mathcal F}) \to \ldots
\]
and its sheaf analogue, denoted by $\underline{{\mathcal C}\textrm{ousin}}_{\{Z\}}({\mathcal F})$. Kempf showed in \cite[Theorem 9.6, 10.3 and 10.5]{Ke1} that under some conditions, these Cousin 
complexes can be used to compute the cohomology groups $H^i(X,{\mathcal F})$:
\begin{proposition}\label{Cousin-Kempf}
Assume that the $Z_i\setminus Z_{i+1}$ are all affine, that $Z_i$ is of codimension $i$ for all $i\geq 0$, and that ${\mathcal F}$ is locally free. Then we have an isomorphism of complexes
\[
\textnormal{Cousin}_{\{Z\}}({\mathcal F}) \simeq H^0(X, \underline{{\mathcal C}\textnormal{ousin}}_{\{Z\}}({\mathcal F}))
\]
and $H^i(X,{\mathcal F})$ is the $i$-th homology group of $\textnormal{Cousin}_{\{Z\}}({\mathcal F})$.
\end{proposition}
We shall now state our main result.
\begin{theorem}\label{cohomology-fibration}
Assume that $S$ is stratified by locally closed affine irreducible schemes $\{Z^1_i\}_{i\in I}$, and that $F$ is stratified by locally closed affine irreducible schemes $\{Z^2_j\}_{j\in J}$ such that:
\item (i) there are affine open subsets $U^1_i\subset S$ and $U^2_j\subset F$ containing respectively $Z^1_i$ and $Z^2_j$, in which they are locally complete intersections.
\item (ii) $E$ is stratified by the $\{Z^1_i\times Z^2_j\}$, and there are open embeddings $U^1_i\times U^2_j\to E$.\\
Then we can compute $H^n(E,{\mathcal L}_1\otimes {\mathcal L}_2)$ with the following converging spectral sequence
\[
E^{p,q}_2=H^p(S,{\mathcal L}_1)\otimes H^q(F,{\mathcal L}_2) \Rightarrow H^*(E,{\mathcal L}_1\otimes {\mathcal L}_2)
\]
If further, either ${\mathcal L}_1$ or ${\mathcal L}_2$ has its cohomology concentrated in a single degree, this spectral sequence collapses, and we get
\[
H^n(E,{\mathcal L}_1\otimes {\mathcal L}_2) = \bigoplus\limits_{p+q=n}H^p(S,{\mathcal L}_1)\otimes H^q(F,{\mathcal L}_2).
\]
\end{theorem}
\begin{proof}
Let us start by fixing some notations. Let 
\[
\overline{S_l}:=\bigcup\limits_{\textrm{codim}(Z^1_i)\geq l} Z^1_i \textrm{, } \overline{F_l}:=\bigcup\limits_{\textrm{codim}(Z^2_j)\geq l} Z^2_j\textrm {, and } \overline{E_l}:=\bigcup\limits_{\textrm{codim}(Z^1_i\times Z^2_j)\geq l} Z^1_i\times Z^2_j
\]
and let $S_l=\overline{S_l}\setminus\overline{S_{l+1}}$, $F_l=\overline{F_l}\setminus\overline{F_{l+1}}$, and $E_l=\overline{E_l}\setminus\overline{E_{l+1}}$.\\
By using twice the excision formula, we can rewrite the cohomology groups occurring in the Cousin complex of ${\mathcal L}_1\otimes {\mathcal L}_2$ in the following way:
\[
\begin{array}{rcll}
H^l_{\overline{E_l}/\overline{E_{l+1}}}(E,{\mathcal L}_1\otimes {\mathcal L}_2)&\simeq &H^l_{E_l}(E\setminus \overline{E_{l+1}},{\mathcal L}_1\otimes {\mathcal L}_2)&\\
&\simeq &\bigoplus\limits_{p+q=l}\bigoplus\limits_{\substack{\textrm{codim}(Z^1_i)=p \\ \textrm{codim}(Z^2_j)=q}} H^l_{Z^1_i\times Z^2_j}(E\setminus \overline{E_{l+1}},{\mathcal L}_1\otimes {\mathcal L}_2)&\\
&\simeq &\bigoplus\limits_{p+q=l}\bigoplus\limits_{\substack{\textrm{codim}(Z^1_i)=p \\ \textrm{codim}(Z^2_j)=q}} H^l_{Z^1_i\times Z^2_j}(U^1_i\times U^2_j,{\mathcal L}_1\otimes {\mathcal L}_2)&\\
&\simeq &\bigoplus\limits_{p+q=l}\bigoplus\limits_{\substack{\textrm{codim}(Z^1_i)=p \\ \textrm{codim}(Z^2_j)=q}} H^l_{Z^1_i\times Z^2_j/\partial (Z^1_i\times Z^2_j)}(E,{\mathcal L}_1\otimes {\mathcal L}_2)&(*)\\
\end{array}
\]
where $\partial (Z^1_i\times Z^2_j)$ denotes $\overline{Z^1_i\times Z^2_j}\setminus Z^1_i\times Z^2_j$. The boundary maps 
\[
\partial_E^l : H^l_{Z^1_i\times Z^2_j/\partial (Z^1_i\times Z^2_j)}(E,{\mathcal L}_1\otimes {\mathcal L}_2) \to H^{l+1}_{Z^1_{i'}\times Z^2_{j'}/\partial (Z^1_{i'}\times Z^2_{j'})}(E,{\mathcal L}_1\otimes {\mathcal L}_2)
\]
are zero whenever $Z^1_{i'}\times Z^2_{j'}\not\subset \partial (Z^1_i\times Z^2_j)$. But a stratum $Z^1_{i'}\times Z^2_{j'}$ of codimension $l+1$ lies $\partial (Z^1_i\times Z^2_j)$ exactly when it is of the form $Z^1_{i'}\times Z^2_j$ or $Z^1_i\times Z^2_{j'}$, with either $Z^1_{i'}\subset\partial Z^1_i$ of codimension codim$(Z^1_i)+1$ or $Z^2_{j'}\subset\partial Z^2_j$ of codimension codim$(Z^2_j)+1$.\\
\indent
We shall now introduce the following bicomplex:
\[
\begin{array}{rcll}
K^{p,q}&:=&\bigoplus\limits_{\substack{\textrm{codim}(Z^1_i)=p \\ \textrm{codim}(Z^2_j)=q}} H^{p+q}_{Z^1_i\times Z^2_j}(U^1_i\times U^2_j,{\mathcal L}_1\otimes {\mathcal L}_2)& \\
&\simeq& \bigoplus\limits_{\substack{\textrm{codim}(Z^1_i)=p \\ \textrm{codim}(Z^2_j)=q}} H^p_{Z^1_i}(U^1_i,{\mathcal L}_1)\otimes H^q_{Z^2_j}(U^2_j,{\mathcal L}_2)& \textrm{by Theorem \ref{kunneth-local}}.
\end{array}
\]
Notice that with the same kind of arguments than for $(*)$, $\bigoplus\limits_{\textrm{codim}(Z^1_i)=p} H^p_{Z^1_i}(U^1_i,{\mathcal L}_1)$ is the $p+1$-th term of 
the $\textrm{Cousin}_{\{S_l\}}({\mathcal L}_1)$, and denote by $\partial^p_{S,i}$ the induced map on $H^p_{Z^1_i}(U^1_i,{\mathcal L}_1)$. For the same reasons, 
$\bigoplus\limits_{\textrm{codim}(Z^2_j)=q} H^q_{Z^2_j}(U^2_j,{\mathcal L}_2)$ is the $(q+1)$-th term of the $\textrm{Cousin}_{\{F_l\}}({\mathcal L}_2)$, 
and denote by $\partial^q_{F,j}$ the induced map on $H^q_{Z^2_j}(U^2_j,{\mathcal L}_2)$. Define boundary maps on $K^{p,q}$ by :
\begin{center}
\begin{tikzcd}
	{} & \vdots 	& \vdots&  \\
\ldots \arrow{r}	&  K^{p,q+1} \arrow{u}\arrow{r}{\partial^{p,q+1}_h}	& K^{p+1,q+1}\arrow{u}\arrow{r}	& \ldots \\
\ldots \arrow{r}	& K^{p,q} \arrow{r}{\partial^{p,q}_h} \arrow{u}{\partial^{p,q}_c} & K^{p+1,q} \arrow{u}{\partial^{p+1,q}_c}\arrow{r}	& \ldots \\
	& \vdots \arrow{u}	&\vdots \arrow{u}  &	
\end{tikzcd}
\end{center}
\item \textbullet the horizontal boundary maps $\partial^{p,q}_h:K^{p,q}\to K^{p+1,q}$ are given by $\bigoplus\limits_{\textrm{codim}(Z^1_i)=p} (\partial^p_{S,i}\otimes 1)$ ;
\item \textbullet the vertical boundary maps $\partial^{p,q}_c:K^{p,q}\to K^{p,q+1}$ are given by $\bigoplus\limits_{\textrm{codim}(Z^2_j)=q} (1\otimes \partial^q_{F,j})$.\\
\indent
The previous argument on the boundary maps $\partial^l_E$ of $\textrm{Cousin}_{\{E_l\}}({\mathcal L}_1\otimes {\mathcal L}_2)$ shows that
\[
Tot(K) = \textrm{Cousin}_{\{E_l\}}({\mathcal L}_1\otimes {\mathcal L}_2)
\]
where $Tot(K)$ denotes the total complex of $K$. Since $K^{p,q}$ vanishes for $p<0$ or $q<0$, the bicomplex $^cH_{*,*}(K)$ obtained from $K$ by taking homology of 
columns gives rise to a spectral sequence whose $E^2$-page is given by
\[
E_{p,q}^2=H_p(^cH_{q,*}(K))
\]
converging to $H_*(Tot(K))$. But by Proposition \ref{Cousin-Kempf}, we have 
\[
H_n(Tot(K))=H_n(\textrm{Cousin}_{\{E_l\}}({\mathcal L}_1\otimes {\mathcal L}_2))=H^n(E,{\mathcal L}_1\otimes {\mathcal L}_2)
\]
and
\[
\begin{array}{rcl}
H_p(^cH_{q,*}(K))&=&H_p(\bigoplus\limits_{\textrm{codim}(Z^1_i)=*}H^*_{Z^1_i}(U^1_i,{\mathcal L}_1)\otimes H^q(F,{\mathcal L}_2))\\
&=&H^p(S,{\mathcal L}_1)\otimes H^q(F,{\mathcal L}_2)
\end{array}
\]
If either ${\mathcal L}_1$ or ${\mathcal L}_2$ has its cohomology concentrated in a single degree, then the last claim is obvious.
\end{proof}

\section{Equivariance}

Let $G_1$ and $G_2$ be two reductive algebraic groups. In this section, we shall take into account actions of $G_1\times G_2$ in the previous results. The key point will come from \cite[Lemma 11.6 ]{Ke1}.\\
\indent
We start to show the following lemma:
\begin{lemma}\label{equivariance-lemma}
With notations and settings of Lemma \ref{isomorphism-flat-modules}, if $G_1$ acts on $X$ in such a 
way that $Z$ is $G_1$-stable, and that both ${\mathcal F}$ and ${\mathcal G}$ are $G_1$-linearized, then
the isomorphisms $t_{\mathcal{F},\mathcal{G}}^p$ given by Lemma \ref{isomorphism-flat-modules} are
$G_1$-equivariant.
\end{lemma}
\begin{proof}
We recall here the $G_1$-linearization of ${\mathcal H}^p_Z({\mathcal F})$ given by Kempf:
Let $\rho : G_1\times X\to X$ be the action map, and let $\pi$ denote the second projection. The $G_1$-linearization of ${\mathcal F}$ gives an isomorphism $\rho^*{\mathcal F}\to \pi^*{\mathcal F}$.
Since $Z$ is $G_1$-stable, we get a map $\rho^*({\mathcal H}^p_Z({\mathcal F}))\to {\mathcal H}^p_{G_1\times Z}(\pi^*{\mathcal F})$, 
and composing it with the isomorphism ${\mathcal H}^p_{G_1\times Z}(\pi^*{\mathcal F}) \to \pi^*({\mathcal H}^p_Z({\mathcal F}))$ gives 
the required natural $G_1$-linearization on ${\mathcal H}^p_Z({\mathcal F})$, see \cite[Lemma 11.3]{Ke1}. The commutativity of the diagram 
(whose horizontal arrows are isomorphisms thanks to the linearizations, and all vertical arrows are isomorphisms by flatness)
\begin{center}
\begin{tikzcd}
\rho^*{\mathcal H}^p_Z({\mathcal F}\otimes{\mathcal G}) \arrow{r}				&{\mathcal H}^p_Z(\pi^*({\mathcal F}\otimes {\mathcal G})) \arrow{r}				&\pi^*{\mathcal H}^p_Z({\mathcal F}\otimes {\mathcal G})\\
\rho^*({\mathcal H}^p_Z({\mathcal F})\otimes{\mathcal G}) \arrow{u}\arrow{r}	&{\mathcal H}^p_Z(\pi^*{\mathcal F})\otimes \pi^*{\mathcal G} \arrow{u}\arrow{r}	&\pi^*({\mathcal H}^p_Z({\mathcal F})\otimes {\mathcal G}) \arrow{u}
\end{tikzcd}
\end{center}
is enough to conclude the lemma.
\end{proof}

Lemma \ref{equivariance-lemma}  is enough to prove equivariant analogues of theorems \ref{kunneth-local} and \ref{cohomology-fibration}. Let us keep the notations and settings of Theorem \ref{kunneth-local}. 
Assume that $G_1$ acts on $X$ and that $Z_1$ is $G_1$-stable, that $G_2$ acts on $Y$ and that $Z_2$ is $G_2$-stable, and that the $L_i$ are 
$G_i$-linearized. 
Equip ${\mathcal L}_1\otimes {\mathcal L}_2$ with the induced $G_1\times G_2$-linearization. Then we have:
\begin{proposition}\label{equivariant-kunneth}
The isomorphism given in Theorem \ref{kunneth-local} is an isomorphism of $G_1\times G_2$-modules.
\end{proposition}
\begin{proof}
By Lemma \ref{equivariance-lemma}, the isomorphism
\[
{\mathcal H}^{p+q}_{Z_1\times Z_2}({\mathcal L}_1\otimes {\mathcal L}_2)\simeq {\mathcal H}^p_{Z_1\times Y}({\mathcal L}_1)\otimes {\mathcal H}^q_{X\times Z_2}({\mathcal L}_2)
\]
obtained in the proof is an isomorphism of $G_1\times G_2$-equivariant sheaves, hence taking global sections, we get an isomorphism of $G_1\times G_2$-modules.
\end{proof}
Now take the notations and settings of Theorem \ref{cohomology-fibration}. As in the introduction, replacing $G_i$ by a finite cover we assume that $G_i$ is factorial for $i=1, 2$. Assume $G_1\times G_2$ acts on $E$, $G_1$ acts on $S$ and trivially on $F$, and $G_2$ acts on $F$ and trivially on $S$, such that the morphism $E\to S$ is $G_2$ invariant and $G_1$ equivariant, and the inclusions of the fibers $F\to E$ are $G_2$-equivariant and $G_1$-invariant. Assume that the subsets $Z^1_i$ and $U^1_i$ are $G_1$-stable, 
that the subsets $Z^2_j$ and $U^2_j$ are $G_2$-stable, and that the ${\mathcal L}_i$ are $G_i$-linearized (hence $G_1\times G_2$-linearized because the action of the other group $G_j$ is trivial). 
Note that we have an exact sequence

\begin{equation}\label{exact-equiv}
0\to \textrm{Pic}^{G_1}(S) \to \textrm{Pic}^{G_1\times G_2}(E) \to \textrm{Pic}^{G_2}(F) \to 0
\end{equation}
since it sits in the following diagram with exact columns whose first and last rows are exact:
\begin{center}
\begin{tikzcd}
			&0 \arrow{d}									&0 \arrow{d}									
			&0 \arrow{d}									&\\
0 \arrow{r}	&{\mathcal X}(G_1) \arrow{d}\arrow{r}			&{\mathcal X}(G_1\times G_2) \arrow{d}{i_E}\arrow{r}	
			&{\mathcal X}(G_2) \arrow{d}{i_F}\arrow{r}		&0\\
			&\textrm{Pic}^{G_1}(S) \arrow{d}\arrow{r}		&\textrm{Pic}^{G_1\times G_2}(E) \arrow{d}{p_E}\arrow{r}
			&\textrm{Pic}^{G_2}(F) \arrow{d}{p_F}			&\\
0 \arrow{r}	&\textrm{Pic}(S) \arrow{d}\arrow{r}				&\textrm{Pic}(E) \arrow{d}\arrow{r}
			&\textrm{Pic}(F) \arrow{d}\arrow{r}				&0\\
			&0	&0	&0	&
\end{tikzcd}
\end{center}
Recall that we have fixed a section $s:\textrm{Pic}(F)\to\textrm{Pic}(E)$. We also have a natural section $s_{\mathcal X}:{\mathcal X}(G_2)\to{\mathcal X}(G_1\times G_2)$ sending a character $\chi$ to the character $\tilde{\chi}(g_1,g_2)=\chi(g_2)$.
Since $\textrm{Pic}(F)$ is projective module, so is $\textrm{Pic}^{G_2}(F)$, and we get a (non-canonical) section
$\textrm{Pic}^{G_2}(F) \to \textrm{Pic}^{G_1\times G_2}(E)$. Let us fix such a section $\tilde{s}$. Then it satisfies $\tilde{s}\circ i_F=i_E\circ s_{\mathcal X}$ and $p_E\circ\tilde{s}=s\circ p_F$. The above short exact sequence splits and we have 
\begin{equation}\label{exact-equiv1}
 \textrm{Pic}^{G_1\times G_2}(E)=\textrm{Pic}^{G_1}(S) \oplus \textrm{Pic}^{G_2}(F).
\end{equation}
 This induces a $G_1\times G_2$-linearization on ${\mathcal L}_1 \otimes {\mathcal L}_2$, and we get:
\begin{proposition}\label{equivariant-cohomology-fibration}
The isomorphism given in Theorem \ref{cohomology-fibration} is an isomorphism of $G_1\times G_2$-modules.
\end{proposition}
\begin{proof}
This is a direct consequence of \cite[Theorem 11.6 (e)]{Ke1}.
\end{proof}
Now assume that $k$ is of characteristic $0$. Then we can lower the hypotheses to obtain a weaker result. 
Namely, we do not need to assume that the closed subschemes occurring in the filtration are stable: in the setting of Theorem \ref{cohomology-fibration}, 
let $\mathcal L$ be a $G_1\times G_2$-linearized line bundle on $E$.

 By (\ref{exact-equiv1}), we can write $\mathcal L= {\mathcal L}_1\otimes {\mathcal L}_2$, where $\mathcal L_1$ is $G_1$-linearizated line bundle on $S$ and $\mathcal L_2$ is $G_2$-linearizated line bundle on $F$. 
 Let $\mathfrak{g}_i$ be the Lie algebra of $G_i$. Then we have:
\begin{proposition}
The isomorphism given in Theorem \ref{cohomology-fibration} is an isomorphism of $\mathfrak{g}_1\times \mathfrak{g}_2$-modules.
\end{proposition}
\begin{proof}
For an affine algebraic group $G$, let $\widehat{G}$ denote the formal completion of $G$ at the identity.
The $G_1\times G_2$-linearization of ${\mathcal L}_1\otimes {\mathcal L}_2$ gives rise to a $\widehat{G_1\times G_2}$-linearization on ${\mathcal L}_1\otimes {\mathcal L}_2$. 
Then by using the natural map $\widehat{G_1}\times \widehat{G_2}\to \widehat{G_1 \times G_2}$, we get $\widehat{G_1}\times \widehat{G_2}$-linearization. 
Hence by \cite[Lemma 11.1(a)]{Ke1}, the cohomology groups $H^{l_1+l_2}_{Z^1_i\times Z^2_j}(U^1_i\times U^2_j, {\mathcal L}_1\otimes{\mathcal L}_2)$, $H^{l_1}_{Z^1_i}(U^1_i,{\mathcal L_1})$ 
and $H^{l_2}_{Z^2_j}(U^2_j,{\mathcal L_2})$ 
have respectively a natural structure of $\widehat{G_1}\times\widehat{G_2}$, $\widehat{G_1}$ and $\widehat{G_2}$-module, that we explicitly describe here.\\
\indent
A $\widehat{G}$-linearization of a sheaf ${\mathcal F}$ of ${\mathcal O}_X$-modules on a scheme $X$ is given by an inverse system of morphisms 
${\mathcal F}\to k[G]/\mathfrak{m}^i\otimes {\mathcal F}$ (where $\mathfrak{m}$ denotes the ideal of regular functions vanishing at the identity). 
Since $k[G]/\mathfrak{m}^i$ are finite dimensional vector spaces and  the local cohomology functor commutes with direct sums, we have that 
 for any closed subsets $Z_2\subset Z_1\subset X$, there is an inverse system $H^p_{Z_1/Z_2}(X,{\mathcal F})\to k[G]/\mathfrak{m}^i\otimes H^p_{Z_1/Z_2}(X,{\mathcal F})$, 
which is precisely the required linearization.\\
\indent
Now since the linearization construction is detailed, it is immediate that the isomorphism
\[
H^{l_1+l_2}_{Z^1_i\times Z^2_j}(U^1_i\times U^2_j, {\mathcal L}_1\otimes{\mathcal L}_2)\to H^{l_1}_{Z^1_i}(U^1_i, {\mathcal L_1})\otimes H^{l_2}_{Z^2_j}(U^2_j, {\mathcal L_2})
\]
is an isomorphism of $\widehat{G_1}\times \widehat{G_2}$-modules.
Notice that for example we can rewrite the $H^{l_1}_{Z^1_i}(U^1_i, {\mathcal L_1})=H^{l_1}_{Z^1_i/\partial Z^1_i}(S, {\mathcal L})$, hence these modules are exactly the direct summands of 
the modules occurring in the Cousin complexes. 
Using \cite[Lemma 11.1(d)]{Ke1}, the Cousin complex of ${\mathcal L}_1\otimes {\mathcal L}_2$ relatively to $\{E_l\}$ is a complex of $\widehat{G_1}\times \widehat{G_2}$-modules. 
Hence the isomorphism given in Theorem \ref{cohomology-fibration} is an isomorphism of $\widehat{G_1}\times \widehat{G_2}$-modules. 
Since we are in characteristic 0, this gives us an isomorphism of Lie algebra $\mathfrak{g}_1\times \mathfrak{g}_2$-modules.
\end{proof}

\section{Application to horospherical varieties}
Let $G$ be a complex connected reductive group. Let $X$ be a smooth toroidal horospherical variety. 
Let $T\subset B$ be a maximal torus of $G$. We need to stratify $X$ in such a way that we can apply Theorem \ref{cohomology-fibration}. 
This is done by looking at the stratification of $X$ by $B\times T'$-orbits. The $B\times T'$-orbits are of the form $BwP\times^P C_{\sigma}$, where $w\in W^P$ 
and $C_{\sigma}$ is the $T'$-orbit of $Y$ associated to a cone $\sigma$ (living in the fan defining $Y$).\\
\indent
$X$ being horospherical and toroidal, we can apply the local structure theorem (Proposition \ref{local_structure}).
Let $U_w=w(w_0^P)^{-1}Bw_0^PP/P\subset G/P$ and 
$U_{\sigma}=\textrm{Spec}(k[M\cap \sigma^{\vee}])\subset Y$. 
Let $p:X\to G/P$ be the projection map. 
Then $BwP\times^P C_{\sigma}$ is closed in $U_wP\times^P U_{\sigma}$, 
which is open in $p^{-1}(U_w)=w(w_0^P)^{-1}.X_0$. 
But $p^{-1}(U_w)\simeq U_w\times Y$ by the local structure theorem. 
Hence the $B\times T'$-orbits are 
isomorphic to $BwP/P \times C_{\sigma}$, and they are closed in an open subset of $X$ of the form 
$U_w\times U_{\sigma}$.
Remark that $U_w$ is $T$-stable, and that $U_{\sigma}$ is $T'$-stable.\\
\indent Both $BwP/P$ and $C_{\sigma}$ are 
locally complete intersections, of respective dimensions $l(w)$ and $\dim(\sigma)$. Assume that $X$ is complete smooth toroidal horospherical variety.
Then we have the action of $G\times T'$ on $X$ such that $X\to G/P$ is $G$-equivariant. By the isomorphism (\ref{equi}), we have $\textrm{Pic}^{G\times T'}(X)\simeq \textrm{Pic}^G(G/P)\oplus \textrm{Pic}^{T'}(Y)$. Hence 
we write any $G$-linearized line bundle $\mathcal L$ on $X$ as $\mathcal L\simeq \mathcal L_1 \otimes \mathcal L_2$, where $\mathcal L_1$ and $\mathcal L_2$ are line bundles on $G/P$ and on $Y$ respectively.
By Borel-Weil-Bott theorem, the cohomology of line bundles on $G/P$ is concentrated in at most one degree. 
Hence we can apply Theorem \ref{cohomology-fibration} to complete toroidal horospherical varieties.  Then,
\begin{corollary}\label{horospherical}
We have an isomorphism of $\mathfrak g \times T'$-modules 
\[
H^n(X, \mathcal L) = \bigoplus\limits_{p+q=n}H^p(G/P,{\mathcal L}_1)\otimes H^q(Y,{\mathcal L}_2)
\]
\end{corollary}
Now let $X$ be a complete horospherical variety. Then it admits a $G$-equivariant complete toroidal 
resolution of singularities $\pi: \tilde{X}\to X$. 
By equivariance, $\tilde{X}$ is also horospherical. Then we can compute with Corollary \ref{horospherical}, 
the cohomology groups of line bundles on $X$ as well.
\begin{proposition}\label{nontoroidalcase}
Let ${\mathcal L}$ be a line bundle on $X$. Then we have $H^i(X,{\mathcal L})=H^i(\tilde{X},\pi^*{\mathcal L})$ for all $i\geq 0$.
\end{proposition}
\begin{proof} First note that, $X$ being a spherical variety it has rational singularities (see for example \cite[Corollary 2.3.4]{perrin2014geometry}).
Then we have $R^q\pi_*\mathcal O_{\tilde X}=0$ for all $q>0$.
By projection formula, we can see that $R^q\pi_*\pi^*{\mathcal L}=0$ for all $q>0$.
Hence the Leray spectral sequence
\[
H^p(X,R^q\pi_*\pi^*{\mathcal L})\Rightarrow H^*(\tilde{X},\pi^*{\mathcal L})
\]
degenerates and we get $H^i(X,{\mathcal L})=H^i(\tilde{X},\pi^*{\mathcal L})$ for all $i\geq 0$.
\end{proof}

{\bf Acknowledgments:} 
We would like to thank Michel Brion for valuable discussions and many critical comments.The
first author would also like to thank Max Planck Institute for Mathematics for the postdoctoral fellowship, and for providing very pleasant hospitality.  
\bibliographystyle{amsalpha}
\bibliography{article}

\providecommand{\bysame}{\leavevmode\hbox to3em{\hrulefill}\thinspace}
\providecommand{\MR}{\relax\ifhmode\unskip\space\fi MR }
\providecommand{\MRhref}[2]{%
  \href{http://www.ams.org/mathscinet-getitem?mr=#1}{#2}
}
\providecommand{\href}[2]{#2}
\begin{thebibliography}{KKLV89}

\bibitem[BP87]{brion1987valuations}
Michel Brion and Franz Pauer, \emph{Valuations des espaces homogenes
  sph{\'e}riques}, Commentarii Mathematici Helvetici \textbf{62} (1987), no.~1,
  265--285.

\bibitem[Bri86]{brion1986quelques}
Michel Brion, \emph{Quelques propri{\'e}t{\'e}s des espaces homogenes
  sph{\'e}riques}, manuscripta mathematica \textbf{55} (1986), no.~2, 191--198.

\bibitem[Bri89]{brion1989groupe}
\bysame, \emph{Groupe de {P}icard et nombres caract{\'e}ristiques des
  vari{\'e}t{\'e}s sph{\'e}riques}, Duke Mathematical Journal \textbf{58}
  (1989), no.~2, 397--424.

\bibitem[Bri90]{Brion1990}
\bysame, \emph{Une extension du th\'{e}or\`{e}me de {B}orel-{W}eil.},
  Mathematische Annalen \textbf{286} (1990), no.~4, 655--660.

\bibitem[Bri94]{brion94}
\bysame, \emph{Representations of reductive groups in cohomology spaces}, Math.
  Ann. \textbf{300} (1994), no.~4, 589--604.

\bibitem[Bri15]{brion-lin}
\bysame, \emph{On linearization of line bundles}, J. Math. Sci. Univ. Tokyo
  \textbf{22} (2015), no.~1, 113--147.

\bibitem[CLS11]{Cox}
David~A. Cox, John~B. Little, and Henry~K. Schenck, \emph{Toric varieties},
  Graduate Studies in Mathematics, vol. 124, American Mathematical Society,
  Providence, RI, 2011.

\bibitem[FI73]{FI}
R.~Fossum and B.~Iversen, \emph{On {P}icard groups of algebraic fibre spaces},
  J. Pure Appl. Algebra \textbf{3} (1973), 269--280.

\bibitem[Ful93]{Ful}
William Fulton, \emph{Introduction to toric varieties}, Annals of Mathematics
  Studies, vol. 131, Princeton University Press, Princeton, NJ, 1993.

\bibitem[Gro63]{EGA3}
Alexander Grothendieck, \emph{{\'E}l\'ements de g\'eom\'etrie alg\'ebrique,
  {III$_B$}}, Publ. Math. IHES \textbf{17} (1963).

\bibitem[Gro68]{SGA2}
\bysame, \emph{Cohomologie locale des faisceaux coh\'erents et th\'eor\`emes de
  {L}efschetz locaux et globaux {$(SGA$} {$2)$}}, Augment{\'e} d'un expos{\'e}
  par Mich{\`e}le Raynaud, S{\'e}minaire de G{\'e}om{\'e}trie Alg{\'e}brique du
  Bois-Marie, 1962, Advanced Studies in Pure Mathematics, Vol. 2.

\bibitem[Kem78]{Ke1}
George~R. Kempf, \emph{The {G}rothendieck-{C}ousin complex of an induced
  representation}, Adv. in Math. \textbf{29} (1978), no.~3, 310--396.

\bibitem[Kem80]{Ke2}
\bysame, \emph{Some elementary proofs of basic theorems in the cohomology of
  quasicoherent sheaves}, Rocky Mountain J. Math. \textbf{10} (1980), no.~3,
  637--645.

\bibitem[KKLV89]{KKLV}
Friedrich Knop, Hanspeter Kraft, Domingo Luna, and Thierry Vust, \emph{Local
  properties of algebraic group actions}, Algebraische {T}ransformationsgruppen
  und {I}nvariantentheorie, DMV Sem., vol.~13, Birkh\"{a}user, Basel, 1989,
  pp.~63--75.

\bibitem[Pas08]{Pa}
Boris Pasquier, \emph{Vari\'et\'es horosph\'eriques de {F}ano}, Bull. Soc.
  Math. France \textbf{136} (2008), no.~2, 195--225.

\bibitem[Per14]{perrin2014geometry}
Nicolas Perrin, \emph{On the geometry of spherical varieties}, Transformation
  Groups \textbf{19} (2014), no.~1, 171--223.

\bibitem[Tch02]{Tchoudjem1}
Alexis Tchoudjem, \emph{Cohomology of line bundles on the wonderful
  compactification of an adjoint semisimple group}, C. R. Math. Acad. Sci.
  Paris \textbf{334} (2002), no.~6, 441--444.

\bibitem[Tch04]{Tchoudjem2}
\bysame, \emph{Cohomology of line bundles over compactifications of reductive
  groups}, Ann. Sci. École Norm. Sup \textbf{(4) 37} (2004), no.~3, 415--448.

\bibitem[Tch07]{Tchoudjem3}
\bysame, \emph{Cohomology of line bundles over wonderful varieties of minimal
  rank}, Bull. Soc. Math. France \textbf{135} (2007), no.~2, 171--214.

\bibitem[Tch10]{Tchoudjem4}
\bysame, \emph{Cohomology with support of line bundles over complete symmetric
  varieties}, Transform. Groups \textbf{15} (2010), no.~3, 655--700.

\bibitem[Tim11]{Ti}
Dmitry~A. Timashev, \emph{Homogeneous spaces and equivariant embeddings},
  xxii+253, Invariant Theory and Algebraic Transformation Groups, 8.

\end{thebibliography}

\end{document}